\documentclass[reqno]{amsart}
\usepackage{amsmath,amsfonts,amssymb,mathrsfs,amsthm}
\usepackage{pifont}
\usepackage[dvipsnames]{xcolor}
\usepackage{hyperref}
\hypersetup{
	colorlinks=true,
	linkcolor=blue,
	citecolor=blue,
	urlcolor=blue
}

\newtheorem{theorem}{Theorem}[section]
\newtheorem{lemma}[theorem]{Lemma}
\newtheorem{proposition}[theorem]{Proposition}
\newtheorem{corollary}[theorem]{Corollary}
\newtheorem{conjecture}[theorem]{Conjecture}

\theoremstyle{definition}
\newtheorem{definition}[theorem]{Definition}

\theoremstyle{remark}
\newtheorem{remark}[theorem]{Remark}

\newtheorem*{maintheorem1}{Theorem 1}
\newtheorem*{introconjA}{Conjecture 1}

\numberwithin{equation}{section}

\newcommand{\g}{\mathfrak{g}}
\newcommand{\B}{\mathcal{B}}
\newcommand{\F}{\mathcal{F}}
\newcommand{\Ext}{\operatorname{Ext}}
\newcommand{\Hom}{\operatorname{Hom}}
\newcommand{\wt}{\operatorname{wt}}

\newcommand{\Tfix}{\mathbb{T}}
\newcommand{\Gl}{\Gamma_\lambda}
\newcommand{\wo}{w_\circ}
\newcommand{\te}{\tilde e}
\newcommand{\tf}{\tilde f}
\newcommand{\chke}{\check e}
\newcommand{\bstar}{b_\star}
\newcommand{\Bl}{\B(\lambda)}
\newcommand{\dom}{\trianglerighteq}
\newcommand{\shape}{\operatorname{sh}}
\newcommand{\rw}{\operatorname{rw}}

\title[Maximal proper extremal subset of a crystal basis]
{Maximal proper extremal subset in the crystal basis of a finite-dimensional irreducible module}

\author{You ZHOU}
\date{}

\begin{document}
\begin{abstract}
    The notion of extremal subsets was introduced by Assaf, Dranowski, and Gonz\'alez as a natural generalization of Demazure crystals. 
    In this paper, we show that the crystal basis of a finite-dimensional irreducible module over a quantum group of finite type has a unique maximal proper extremal subset, whose complement is the set of elements that generate the whole crystal basis as an extremal subset.
	We conjecture that this complement consists of all nonzero elements obtained by successively applying the lowering Kashiwara operators to a distinguished element,
    which is constructed from the lowest element of the crystal basis.
	We prove this conjecture in type $A$.
\end{abstract}
	\maketitle

   \section{Introduction}

 In \cite{kashiwara1993}, Kashiwara proved that
the crystal basis of a Demazure module (in an integrable highest weight module)
is realized as a subset, called a Demazure crystal, of the crystal basis of
the integrable highest weight module,
and that Demazure crystals satisfy the string property;
see also Joseph \cite{joseph2003}.
Assaf, Dranowski, and Gonz\'alez \cite{assaf2023} introduced extremal subsets as a natural generalization of Demazure crystals;
a subset of a crystal basis is called an extremal subset if it satisfies the string property.
They used extremal subsets to determine whether the tensor product of two Demazure crystals decomposes
into a disjoint union of Demazure crystals.
Recently, Assaf and Gonz\'alez \cite{assafgonzalez2025} characterized unions of
Demazure crystals as extremal subsets satisfying
certain additional local conditions, and showed that such unions
decompose into Demazure atoms,
and have the same characters as dual Polo modules.

Extremal subsets of the crystal basis $\Bl$ of dominant highest weight $\lambda$ are closed under unions and intersections \cite{armon2023}.
In particular, for $b\in\Bl$, the intersection of all extremal subsets containing $b$ is the smallest extremal subset $\langle b\rangle_{\Ext}$ containing $b$.
Moreover, since the lowest element $b_{\wo\lambda}\in \Bl$ satisfies $\langle b_{\wo\lambda}\rangle_{\Ext}=\Bl$,
the union $E_{\max}$ of all proper extremal subsets of $\Bl$ is the unique maximal proper extremal subset of $\Bl$.
In this paper, we conjecture that $\Bl\setminus E_{\max}$ consists of all nonzero elements
obtained by successively applying the lowering Kashiwara operators to a distinguished element,
and we prove this conjecture in type $A$.

Let us explain our main result more precisely.
Let $\g$ be a finite-dimensional simple Lie algebra with $I$ the index set of its simple roots,
and $U_q(\g)$ the quantum group associated to $\g$.
Denote by $\Bl$ the crystal basis of the finite-dimensional irreducible $U_q(\g)$-module of dominant highest weight $\lambda$,
with $\te_i,\tf_i$ ($i\in I$) the Kashiwara operators.
For $b\in\Bl$ and $i\in I$, the $i$-string of $b$ is defined to be
\[
S^i(b):=\{\te_i^{\max}(b),\dots,\te_i(b),b,\tf_i(b),\dots,\tf_i^{\max}(b)\}\setminus\{0\},
\]
where $\te_i^{\max}(b):=\te_i^{k}(b)$ for the largest $k\ge0$ with $\te_i^k(b)\ne0$,
and $\tf_i^{\max}(b):=\tf_i^{k}(b)$ for the largest $k\ge0$ with $\tf_i^k(b)\ne0$.
A subset $E\subseteq\Bl$ is called an extremal subset if
$E\cap S^i(b)\in\{\varnothing,\{\te_i^{\max}(b)\},S^i(b)\}$
for all $b\in\Bl$ and $i\in I$.
Set $\Gl:=\{c\in\Bl\mid \langle c\rangle_{\Ext}=\Bl\}$.
We prove in Proposition~\ref{prop:max-proper-extremal} that $E_{\max}=\Bl\setminus\Gl$.
Thus we can describe $E_{\max}$ in terms of $\Gl$.

To describe $\Gl$, we introduce the \emph{checked operators} $\chke_i$ ($i\in I$) on $\Bl$ as follows:
\[
\chke_i(b):=
\begin{cases}
	\tf_i\te_i^{\max}(b) & \text{if } \te_i(b)\ne0,\\
	b & \text{if } \te_i(b)=0.
\end{cases}
\]
Let $w_\circ=s_{i_1}\cdots s_{i_\ell}$ be a reduced expression of the longest element $w_\circ$
of the Weyl group $W$, and set $\chke_{w_\circ}:=\chke_{i_1}\cdots\chke_{i_\ell}$;
the operator $\chke_{w_\circ}$ is independent of the choice of a reduced expression of $w_\circ$ (see Definition~\ref{def:chkw}).
For $b\in\Bl$, we define
\[
\mathcal{C}_b:=\{c\in\Bl\mid \chke_{w_\circ}(c)=\chke_{w_\circ}(b)\}.
\]
Set $\bstar:=\chke_{w_\circ}(b_{\wo\lambda})$, and $\F(\bstar):=\{\tf_{j_1}\cdots\tf_{j_s}(\bstar)\mid s\ge0, j_1,\dots,j_s\in I\}\setminus\{0\}$.

\begin{introconjA}[Conjecture~\ref{conj:main1-all-types}]
	We have $\Gl=\F(\bstar)=\mathcal{C}_{\bstar}$.
\end{introconjA}

\begin{maintheorem1}[Theorem~\ref{thm:main1}]
	Conjecture~1 is true if $\g$ is of type $A$.
\end{maintheorem1}

In Appendix~\ref{sec:applications}, we explore further properties of $\mathcal{C}_b$, $b\in\Bl$.
In Proposition~\ref{prop:decomp-into-classes},
we show that every Demazure atom and every
extremal subset in $\Bl$ decomposes into disjoint unions of
$\mathcal{C}_{b}$'s for $b \in B(\lambda)$;
in particular, this decomposition refines the
decomposition of a Demazure crystal into Demazure atoms.
Let $\rho$ denote the sum of all fundamental weights,
and $b_\rho$ the highest element of $\B(\rho)$.
We also show that $b_\rho\otimes\mathcal{C}_b$ is ``isomorphic'' to a Demazure crystal
in the sense of Proposition~\ref{prop:brho-C_b}.

This paper is organized as follows.
In Section~\ref{sec:preliminaries}, we fix our notation.
In Section~\ref{sec:extremal-subsets}, we recall extremal subsets and introduce the checked operators.
In Section~\ref{sec:main-theorems}, we state Conjecture~1 and Theorem~1.
In Section~\ref{sec:plactic}, we review the plactic monoid.
In Section~\ref{sec:greene}, we prove a combinatorial result on the Greene invariants of words,
which is used in Section~\ref{sec:proof-thm1} to prove Theorem~1.
In Appendix~\ref{sec:applications}, we give further properties of the subsets $\mathcal{C}_b$ for $b\in\Bl$.

%%%%%%%%%%%%%%%%%%%%%%%%%%%%%%%%%%%%%%%%%%%%%%%%%%%%%%%%%%%%%%%%
	\section{Basic notation}\label{sec:preliminaries}
	%%%%%%%%%%%%%%%%%%%%%%%%%%%%%%%%%%%%%%%%%%%%%%%%%%%%%%%%%%%%%%%%
	
	Let $\mathfrak{g}$ be a finite-dimensional simple Lie algebra over $\mathbb{C}$, 
	and $\mathfrak{h}$ a Cartan subalgebra of $\mathfrak{g}$.
	Set $\mathfrak{h}^*:=\Hom_\mathbb{C}\left(\mathfrak{h},\mathbb{C}\right)$, 
	and denote by $\langle \cdot,\cdot\rangle:\mathfrak{h}^\ast\times\mathfrak{h}\to\mathbb{C}$ the canonical pairing.
	Let
	$\Delta=\{\alpha_i\}_{i\in I}\subset\mathfrak{h}^*$ be the set of simple roots, and
	$\Delta^\lor=\{\alpha_i^\lor\}_{i\in I}\subset\mathfrak{h}$ the set of simple coroots.
	We set $Q:=\bigoplus_{i\in I}\mathbb{Z}\alpha_i$ and $Q^+:=\sum_{i\in I}\mathbb{Z}_{\ge 0}\alpha_i$. 
	Let $P=\bigoplus_{i\in I}\mathbb{Z}\varpi_i\subseteq\mathfrak{h}^*$ be the weight lattice of $\mathfrak{g}$, where $\varpi_i$ ($i\in I$) are the fundamental weights,
	and let $P^+=\sum_{i\in I}\mathbb{Z}_{\ge 0}\varpi_i\subset P$ be the set of dominant integral weights.
	We set $\rho:=\sum_{i\in I}\varpi_i\in P^+$.
	For $\mu,\nu\in P$,
	we write $\mu\ge\nu$ if $\mu-\nu\in Q^+$;
	we call this partial order the \emph{dominance order} on $P$.
	Let $W=\langle s_i \mid{i\in I}\;\rangle$ be the Weyl group of $\g$, where $s_i$ is the simple reflection in $\alpha_i$.
	Denote by $w_\circ$ the longest element of $W$.
	
	Let $U_q(\g)$ be the quantized enveloping algebra of $\g$ over $\mathbb{C}(q)$
	(the quantum group). For $\lambda \in P^{+}$, let $V(\lambda) = \bigoplus_{\mu \in P} V(\lambda)_\mu$ denote the (finite-dimensional)
	irreducible highest weight $U_q(\g)$-module of highest weight $\lambda$,
	where $V(\lambda)_\mu$ denotes the weight space of weight $\mu$.
	
	A \emph{crystal} is a set $\B$ together with maps
	$\wt:\B\to P$, the \emph{Kashiwara operators}
	$\te_i,\tf_i:\B\to\B\sqcup\{0\}$, and
	$\varepsilon_i,\varphi_i:\B\to\mathbb{Z}\cup\{-\infty\}$ $(i\in I)$
	satisfying certain conditions;
	see \cite[Definition 4.5.1]{hongkong}.
	All crystals in this paper are seminormal, that is,
	$\varepsilon_i(b)=\max\{k\ge0\mid\te_i^k(b)\ne0\}$ and
	$\varphi_i(b)=\max\{k\ge0\mid\tf_i^k(b)\ne0\}$ for $b\in\B$ and $i\in I$;
	in this case, we define $
	\tilde{e}_i^{\max}(b) := \tilde{e}_i^{\,\varepsilon_i(b)}(b)$ and $
	\tilde{f}_i^{\max}(b) := \tilde{f}_i^{\,\varphi_i(b)}(b)$
	for $b\in \B$ and $i\in I$.
	The \emph{crystal graph} of $\B$ is the $I$-labeled directed graph whose vertex set is $\B$ and whose edges are given as follows:
	for $b,b'\in\B$ and $i\in I$,
	we have an edge $b\xrightarrow{\,i\,}b'$ if $\tf_i(b)=b'$;
	a \emph{connected
		component} of $\B$ means the vertex set of a connected component of this graph.
	For a crystal $\B$ and $x\in\B$, we denote by $C(x)$ the connected component of $\B$
	containing $x$.
	For crystals $\B$ and $\B'$,
	we endow the set $\B\otimes\B':=\B\times\B'$ with a crystal structure as in \cite[\S 4.4]{hongkong},
	where we write $x\otimes y$ for $(x,y)\in\B\times\B'$.
	
	An element $b$ of a crystal $\B$ is called a \emph{highest element} (resp., \emph{lowest element}) of $\B$
	if $\te_i(b)=0$ (resp., $\tf_i(b)=0$) for all $i\in I$.

	For $\lambda\in P^+$,
	let $\Bl$ be the crystal basis of the $U_q(\g)$-module $V(\lambda)$,
	and $b_\lambda\in\Bl$ (resp., $b_{w_\circ\lambda}\in\Bl$) the highest (resp., lowest) element of weight $\lambda$ (resp., $w_\circ\lambda$).

	\begin{definition}[$i$-string]\label{def:i-string}
		For $b\in\Bl$ and $i\in I$, the \emph{$i$-string} of $b$ is the subset
		\[
		S^i(b):=\{\te_i^{\max}(b),\dots,\te_i(b),b,\tf_i(b),\dots,\tf_i^{\max}(b)\}\setminus\{0\}.
		\]
	\end{definition}
	
	Let $w\in W$.
	If $w=s_{i_1}\cdots s_{i_\ell}$ is a reduced expression of $w$, 
	then we set $\te_{w}^{\max}:=\te_{i_1}^{\max}\cdots\te_{i_\ell}^{\max}$.
	We deduce the following proposition from \cite[\S 7.2]{kashiwara1994}.
	
	\begin{proposition}\label{prop:eimax-indep}
		For $w\in W$, the operator $\te_{w}^{\max}$ is independent of the choice of a
		reduced expression of $w$.
		Also, we have $\te_{w_\circ}^{\max}\B(\lambda)=\{b_\lambda\}$.
	\end{proposition}
	%%%%%%%%%%%%%%%%%%%%%%%%%%%%%%%%%%%%%%%%%%%%%%%%%%%%%%%%%%%%%%%%
	\section{Extremal subsets and checked operators}\label{sec:extremal-subsets}
	%%%%%%%%%%%%%%%%%%%%%%%%%%%%%%%%%%%%%%%%%%%%%%%%%%%%%%%%%%%%%%%%
	
	\subsection{Extremal subsets}\label{subsec:extremal-crystals}
	Fix $\lambda\in P^+$.
	The notion of an extremal subset is a natural
	generalization of a Demazure crystal
	$\B_w(\lambda)\subseteq\Bl$, $w\in W$,
	introduced by Kashiwara in \cite{kashiwara1993};
    we know from \cite[Proposition 3.3.5]{kashiwara1993} that a Demazure crystal is an extremal subset.
	
	\begin{definition}[{\cite[\S 4]{assaf2023}}]\label{def:extremal-recall}
		A subset $X\subseteq\Bl$ is called an \emph{extremal} subset 
		if
		\begin{align}\label{i-stringproperty}
			S^i(b)\cap X\in\{\varnothing,\ S^i(b),\;\{\te_i^{\max}(b)\}\}
		\end{align}
		for all $b\in\Bl$ and $i\in I$.
	\end{definition}

	We denote by $\mathbb{E}_\lambda$ the set of all extremal subsets of $\Bl$.
	
	\begin{proposition}[{\cite[\S 4]{assaf2023}}]\label{prop:ext-e-stable}
		If $X\in\mathbb{E}_\lambda$, then $X$ is $\te_i$-stable for all $i\in I$,
		that is, $\te_iX\subseteq X\sqcup\{0\}$ for all $i\in I$.
	\end{proposition}
	
	\begin{proposition}[{\cite[Lemma 4.1]{armon2023}}]\label{prop:cup-cap-extremal}
		If $X,Y\in\mathbb{E}_\lambda$, then $X\cap Y,\;X\cup Y\in\mathbb{E}_\lambda$.
	\end{proposition}
	
	\begin{definition}\label{def:ext-generated}
		Let $Y$ be a subset of $\Bl$.
		The \emph{extremal subset generated by $Y$} is defined to be
		\[
		\langle Y\rangle_{\Ext}:=\bigcap_{\substack{X\in\mathbb{E}_\lambda\\Y\subseteq X}}X.
		\]
		For $b\in\Bl$,
		we write $\langle\{b\}\rangle_{\Ext}$ as $\langle b\rangle_{\Ext}$ for simplicity of notation.
	\end{definition}
	
	We set
	\begin{align}\label{eq:Gl}
		\Gl:=\{b\in\Bl\mid \langle b\rangle_{\Ext}=\Bl\}.
	\end{align}
	
	\begin{remark}\label{rem:bwo-in-Gl}
		Every element of $\Bl$ is obtained from the lowest element $b_{w_\circ\lambda}$ by successively applying
		the operators $\te_i$ ($i\in I$).
		Since extremal subsets are $\te_i$-stable for all $i\in I$ (Proposition~\ref{prop:ext-e-stable}),
		any extremal subset containing $b_{w_\circ\lambda}$ contains all elements of $\Bl$.
		Therefore the lowest element $b_{w_\circ\lambda}\in\Bl$ is contained in $\Gl$.
	\end{remark}

    \begin{proposition}\label{prop:Emaxexists}
        The crystal basis $\Bl$ has a unique maximal proper extremal subset $E_{\max}$.
    \end{proposition}
    \begin{proof}
        Let $E_{\max}$ denote the union of all proper extremal subsets of $\Bl$; 
        $E_{\max}\in\mathbb{E}_\lambda$ by Proposition~\ref{prop:cup-cap-extremal}.
        Because $b_{w_\circ\lambda}$ is contained in no proper extremal subset of $\Bl$ (Remark~\ref{rem:bwo-in-Gl}),
        we have $b_{w_\circ\lambda}\notin E_{\max}$, and hence $E_{\max}\subsetneq\Bl$.
        Thus $E_{\max}$ is a proper extremal subset containing every proper extremal subset of $\Bl$,
        that is, the unique maximal proper extremal subset of $\Bl$.
    \end{proof}

	\begin{proposition}\label{prop:max-proper-extremal}
		We have $\Bl\setminus\Gl=E_{\max}$.
	\end{proposition}
	\begin{proof}
        It suffices to show that $\Bl\setminus\Gl$ is identical to $E_{\max}$, the union of all proper extremal subsets of $\Bl$;
		see Proposition~\ref{prop:Emaxexists}.
		Let $X$ be a proper extremal subset of $\Bl$.
		Because $\langle b\rangle_{\Ext}\subseteq X\subsetneq \Bl$ for $b\in X$, we deduce that $b\notin \Gl$, or equivalently,
		$b\in \Bl\setminus\Gl$.
		Therefore, we get $\Bl\setminus\Gl\supseteq X$,
		and hence $\Bl\setminus\Gl\supseteq E_{\max}$.
		Let us show the reverse inclusion $\Bl\setminus\Gl\subseteq E_{\max}$.
		If $b\in\Bl\setminus\Gl$,
		then we have $\langle b\rangle_{\Ext}\subsetneq \Bl$.
		Thus, we get $b\in\langle b\rangle_{\Ext}\subseteq E_{\max}$.
		Therefore, $\Bl\setminus\Gl\subseteq E_{\max}$.
		This completes the proof of the proposition.
	\end{proof}
	
	Let $\mathbf{I}:=\{\mathbf{i}=(i_1,\dots,i_s)\mid s\ge 0,\;i_1,\dots,i_s\in I\}$ be the set of finite sequences of elements of $I$.
	For $\mathbf{i}=(i_1,\dots,i_s)\in\mathbf{I}$,
	we set $\tf_{\mathbf{i}}:=\tf_{i_1}\cdots\tf_{i_s}$.
	We define $\F:=\{\tf_{\mathbf{i}}\mid\mathbf{i}\in\mathbf{I}\}$,
	and $\F(Y):=\{f(b)\mid f\in\F,\ b\in Y\}\setminus\{0\}$ for a subset $Y\subseteq\Bl$;
	for $b\in\Bl$, we write $\F(b)$ for $\F(\{b\})$.
	
	\begin{proposition}\label{prop:Gl-F-stable}
		The subset $\Gl$ is $\F$-stable,
		that is, $\F(\Gl)\subseteq\Gl$.
	\end{proposition}
	\begin{proof}
		It suffices to show that $\Gl$ is $\tf_i$-stable for all $i\in I$,
		that is, $\tf_i\Gl\subseteq\Gl\sqcup\{0\}$ for all $i\in I$.
		Let $b\in\Gl$ and $i\in I$ be such that $\tf_i(b)\ne0$.
		Since extremal subsets are $\te_i$-stable, we see that $b=\te_i\tf_i(b)\in\langle\tf_i(b)\rangle_{\Ext}$,
		which implies that $\Bl=\langle b\rangle_{\Ext}\subseteq\langle\tf_i(b)\rangle_{\Ext}.$
		Therefore, we get $\langle\tf_i(b)\rangle_{\Ext}=\Bl$,
		and hence $\tf_i(b)\in\Gl$.
	\end{proof}
	
	\subsection{Checked operators}\label{subsec:checked-operator}
	
	\begin{definition}\label{def:chke}
		For $b\in\Bl$ and $i\in I$, we set
		\begin{align}
			\chke_i(b):=
			\begin{cases}
				\tf_i\te_i^{\max}(b) & \text{if }\te_i(b)\ne0,\\
				b & \text{if }\te_i(b)=0.
			\end{cases}
		\end{align}
	\end{definition}
	\begin{remark}\label{rem:chkechke=chke}
		We see that $(\chke_i)^2=\chke_i$ for $i\in I$.
	\end{remark}
	
	\begin{proposition}\label{prop:chke-invariant}
		For $b\in\Bl$ and $i\in I$,
		we have $\langle\chke_i(b)\rangle_{\Ext}=\langle b\rangle_{\Ext}$.
	\end{proposition}
	\begin{proof}
		We first show the inclusion $\langle\chke_i(b)\rangle_{\Ext}\subseteq\langle b\rangle_{\Ext}$.
		Since $\langle b\rangle_{\Ext}$ is $\te_i$-stable (Proposition~\ref{prop:ext-e-stable}), 
		we have $\chke_i(b)=\te_i^k(b)\in\langle b\rangle_{\Ext}$ for some $k\in\mathbb{Z}_{\ge 0}$,
		which implies
		$\langle\chke_i(b)\rangle_{\Ext}\subseteq\langle b\rangle_{\Ext}$.
		We next show the reverse inclusion $\langle\chke_i(b)\rangle_{\Ext}\supseteq\langle b\rangle_{\Ext}$.
		If $\chke_{i}(b)=b$,
		then the assertion is obvious.
		Assume that $\chke_{i}(b)\ne b$.
		We note that $S^i(\chke_{i}(b))=S^i(b)$.
		Since $\chke_i(b)\ne b$, Definition~\ref{def:chke} forces $\te_i(b)\ne0$,
		and hence $\te_i^{\max}(b)\ne\chke_{i}(b)$.
		Because $\chke_{i}(b)\in S^i(\chke_{i}(b))\cap\langle\chke_i(b)\rangle_{\Ext}=S^i(b)\cap\langle\chke_i(b)\rangle_{\Ext}$,
		it follows that
		$S^i(b)\cap\langle\chke_i(b)\rangle_{\Ext}=S^i(b)$.
		Thus we get $b\in\langle\chke_i(b)\rangle_{\Ext}$,
		and hence
		$\langle b\rangle_{\Ext}\subseteq\langle\chke_i(b)\rangle_{\Ext},$
		as desired.
	\end{proof}
	
	Recall that $\rho:=\sum_{i\in I}\varpi_i$.
	
	\begin{proposition}\label{prop:eimax-to-chkw}
		For $b\in\Bl$ and $i\in I$, we have
		$
		\te_i^{\max}(b_\rho\otimes b)=b_\rho\otimes\chke_i(b).
		$
	\end{proposition}
	\begin{proof}
		We see that
		\[
		\te_i(b_\rho\otimes b)=\begin{cases}0 &\text{if } \varepsilon_i(b)\le 1,\\
			b_\rho\otimes \te_i(b) &\text{if }  \varepsilon_i(b)>1.\end{cases}
		\]
		Also, we have the following:
		\begin{enumerate}
			\item[(i)] If $\varepsilon_i(b)\le 1$,
			then $\te_i(b_\rho\otimes b)=0$ and $\chke_{i}(b)=b$.
			Hence, $\te_i^{\max}(b_\rho\otimes b)=b_\rho\otimes b=b_\rho\otimes \chke_{i}(b)$.
			\item[(ii)] If $\varepsilon_i(b)=k> 1$,
			then $\chke_{i}(b)=\tf_i\te_i^{\max}(b)=\te_i^{k-1}(b)$.
			Also, since $\varepsilon_i(b_\rho)=0$ and $\langle\wt(b_\rho),\alpha_i^\vee\rangle=\langle\rho,\alpha_i^\vee\rangle=1$,
			by the tensor product rule (see \cite[\S 4.4]{hongkong})
			we have 
			\begin{align*}
				\varepsilon_i(b_\rho\otimes b)=\max\{\varepsilon_i(b_\rho),\varepsilon_i(b)-\langle\wt(b_\rho),\alpha_i^\vee\rangle\}=\max\{0,k-1\}=k-1.
			\end{align*}
			Hence, we obtain $$\te_i^{\max}(b_\rho\otimes b)=\te_i^{k-1}(b_\rho\otimes b)
			=b_\rho\otimes \te_i^{k-1}(b)=b_\rho\otimes \tf_i\te_i^{k}(b)=
			b_\rho\otimes \chke_{i}(b).$$
		\end{enumerate}
		Thus we have proved the proposition.
	\end{proof}
	
	\begin{definition}\label{def:chkw}
		Let $w\in W$, and let $w=s_{i_1}\cdots s_{i_\ell}$ be a reduced expression of $w$.
		Then we set $\chke_{w}:=\chke_{i_1}\cdots\chke_{i_\ell}$;
		we deduce from Propositions~\ref{prop:eimax-to-chkw} and \ref{prop:eimax-indep}
		that the operator $\chke_{w}$ is independent of the choice of a reduced
		expression of $w$.
	\end{definition}
	
	Recall that $w_\circ$ denotes the longest element of $W$.
	
	\begin{proposition}\label{prop:chkwmax-check-stable}
		For $i\in I$, we have $\chke_i\chke_{w_\circ}=\chke_{w_\circ}$ and $\chke_{w_\circ}\chke_i=\chke_{w_\circ}$.
		Therefore we have $(\chke_{w_\circ})^2=\chke_{w_\circ}$.
	\end{proposition}
	\begin{proof}
		Let $i\in I$.
		Since $\ell(\wo s_i)=\ell(\wo)-1$,
		$\wo$ has a reduced expression
		$\wo=s_{i_1}\cdots s_{i_{\ell-1}}s_i$ ending with $s_i$.
		By Remark~\ref{rem:chkechke=chke},
		we get
		$$\chke_{w_\circ}\chke_i
		=\chke_{i_1}\cdots\chke_{i_{\ell-1}}\chke_i\chke_i
		=\chke_{i_1}\cdots\chke_{i_{\ell-1}}\chke_i=\chke_{w_\circ}.$$
		Also, since $\wo^{-1}=\wo$,
		it follows that $\wo=s_i s_{i_{\ell-1}}\cdots s_{i_1}$ is a reduced expression of $w_\circ$.
		By the same argument as above, we obtain $\chke_i\chke_{w_\circ}=\chke_{w_\circ}$, as desired.
	\end{proof}
	
	For $b\in\Bl$, we set
	\[
	\mathcal{C}_b:=\{c\in\Bl\mid\chke_{\wo}(c)=\chke_{\wo}(b)\}.
	\]
	We also set $\Tfix:=\{c\in\Bl\mid \chke_{\wo}(c)=c\}$.
	
	\begin{proposition}\label{prop:unique-checked-rep}
		We have
		$
		\Bl=\bigsqcup_{b\in\Tfix}\mathcal{C}_b.
		$		
		Moreover, $\mathcal{C}_b\cap\Tfix=\{\chke_{\wo}(b)\}$ for $b\in\Bl$.
	\end{proposition}
	\begin{proof}
		Define a relation $\sim$ on $\Bl$ by:
		for $b,c\in\Bl$, $b\sim c$ if $c\in \mathcal{C}_b$.
		It can be easily seen that $\sim$ is an equivalence relation on $\Bl$ whose equivalence classes are
		$\{\mathcal{C}_b\mid b\in\Bl\}$;
		if $c\in\mathcal{C}_b\cap\mathbb{T}$,
		then we have $c=\chke_{\wo}(c)=\chke_{\wo}(b)$, which shows $\mathcal{C}_b\cap\Tfix\subseteq\{\chke_{\wo}(b)\}$.
		Conversely, by Proposition~\ref{prop:chkwmax-check-stable} we have $\chke_{\wo}(\chke_{\wo}(b))=\chke_{\wo}(b)$,
		so that $\chke_{\wo}(b)\in\Tfix$ and $\chke_{\wo}(b)\in\mathcal{C}_b$.
		Thus we get $\mathcal{C}_b\cap\Tfix=\{\chke_{\wo}(b)\}$.
		Because $\mathbb{T}$ is the set of all representatives for $\sim$,
		we have $\Bl=\bigsqcup_{b\in\Tfix}\mathcal{C}_b$.
		Thus we have proved the proposition.
	\end{proof}

	\begin{proposition}\label{prop:weak-string-property}
		For $a,b\in\Bl$ and $i\in I$, we have
		\[
		\mathcal{C}_a\cap S^i(b)\in\left\{\varnothing,\ S^i(b),\
		S^i(b)\setminus\{\te_i^{\max}(b)\},\ \{\te_i^{\max}(b)\}\right\}.
		\]
	\end{proposition}
	\begin{proof}
		If $|S^i(b)|=1$, then $S^i(b)=\{b\}=\{\te_i^{\max}(b)\}$.
		It is obvious that
		$\mathcal{C}_a\cap S^i(b)\in\{\varnothing,\ S^i(b)\}$.
		Assume that $|S^i(b)|>1$,
		and $\mathcal{C}_a\cap S^i(b)\notin\{\varnothing,\ \{\te_i^{\max}(b)\}\}$.
		It suffices to show that 
		\begin{align}\label{eq:string-below-top}
			S^i(b)\setminus\{\te_i^{\max}(b)\}\subseteq\mathcal{C}_a\cap S^i(b).
		\end{align}
		Let $c\in \mathcal{C}_a\cap S^i(b)$ be such that $c\ne\te_i^{\max}(b)$,
		and let $d\in S^i(b)\setminus\{\te_i^{\max}(b)\}$.
		Then we have $\chke_{i}(d)=\chke_{i}(c)$.
		Thus, by Proposition~\ref{prop:chkwmax-check-stable},
		we have $\chke_{w_\circ}(d)=\chke_{w_\circ}\chke_{i}(d)=\chke_{w_\circ}\chke_{i}(c)=\chke_{w_\circ}(c)=\chke_{w_\circ}(a)$,
		which implies that $d\in\mathcal{C}_a$.
		Thus we have proved \eqref{eq:string-below-top},
		thereby completing the proof of the proposition.
	\end{proof}
	
	%%%%%%%%%%%%%%%%%%%%%%%%%%%%%%%%%%%%%%%%%%%%%%%%%%%%%%%%%%%%%%%%
	\section{Main result}\label{sec:main-theorems}
	%%%%%%%%%%%%%%%%%%%%%%%%%%%%%%%%%%%%%%%%%%%%%%%%%%%%%%%%%%%%%%%%
	Fix $\lambda\in P^+$.
	We set $\bstar:=\chke_{\wo}(b_{\wo\lambda})$;
	recall that $b_{\wo\lambda}\in\Bl$ is the lowest element.
	Note that $\mathcal{C}_{b_{\star}}=\mathcal{C}_{b_{\wo\lambda}}$ by Proposition~\ref{prop:chkwmax-check-stable}.
	
	\begin{conjecture}[$=$ Conjecture 1]\label{conj:main1-all-types}
		For $\lambda\in P^+$, we have
		$
		\Gl=\F(\bstar)=\mathcal{C}_{b_{\star}}.
		$
		Therefore, by Proposition~\ref{prop:max-proper-extremal}, $\Bl\setminus\F(\bstar)$ is the maximal proper extremal subset
		of $\Bl$.
	\end{conjecture}
	
	\begin{proposition}\label{prop:reduction}
		If $\F(\mathcal{C}_{b_\star})\subseteq\mathcal{C}_{b_\star}$,
		then Conjecture~\ref{conj:main1-all-types} is true.
	\end{proposition}
	\begin{proof}
		First, we show that $\mathcal{C}_{\bstar}=\F(\bstar)$.
		By the assumption, we have $\F(\bstar)\subseteq \mathcal{C}_{\bstar}.$
		Let us show the reverse inclusion.
		Let $b\in \mathcal{C}_{b_{\star}}$.
		Then, $b_{\star}=\chke_{w_\circ}(b_{\star})=\chke_{w_\circ}(b)=\te_{i_1}\cdots\te_{i_m}(b)$
		for some $m\ge0$ and $i_1,\dots,i_m\in I$.
		Then we get $b=\tf_{i_m}\cdots\tf_{i_1}(b_\star)\in \F(b_\star)$, as desired.
		
		Next, we show that $\mathcal{C}_{\bstar}=\Gl$. 
		Let $w_\circ=s_{j_1}\cdots s_{j_\ell}$ be a reduced expression of $w_\circ$.
		We deduce from Proposition~\ref{prop:chke-invariant} that
		\[
		\langle\bstar\rangle_{\Ext}=
		\langle\chke_{w_\circ}(b_{w_\circ\lambda})\rangle_{\Ext}=
		\langle\chke_{j_1}\cdots\chke_{j_\ell}(b_{w_\circ\lambda})\rangle_{\Ext}=
		\langle b_{\wo\lambda}\rangle_{\Ext}=\Bl,
		\]
		which implies that $\bstar\in\Gl$.
		Because $\mathcal{C}_{\bstar}=\F(\bstar)$ as seen above, and $\Gl$ is $\F$-stable by Proposition~\ref{prop:Gl-F-stable},
		it follows that $\mathcal{C}_{\bstar}=\F(\bstar)\subseteq\Gl$.
		Let us show the reverse inclusion $\Gl\subseteq \mathcal{C}_{\bstar}$.
		It follows from Proposition~\ref{prop:weak-string-property} that
		$\mathcal{C}_{\bstar}\cap S^i(b)\in\{\varnothing,\;S^i(b),\;
		S^i(b)\setminus\{\te_i^{\max}(b)\},\;\{\te_i^{\max}(b)\}\}.$
		Because $\mathcal{C}_{\bstar}$ is $\tf_i$-stable by the assumption,
		we deduce that 
		$\mathcal{C}_{\bstar}\cap S^i(b)\in\{\varnothing,\; S^i(b),\;S^i(b)\setminus\{\te_i^{\max}(b)\}\},$
		which implies that
		$\Bl\setminus \mathcal{C}_{\bstar}$ is an extremal subset of $\Bl$.
		Hence, by Proposition~\ref{prop:max-proper-extremal}, we obtain $\Gl\subseteq \mathcal{C}_{\bstar}$, as desired.
		Thus we have shown $\Gl=\mathcal{C}_{\bstar}$,
		thereby completing the proof of the proposition.
	\end{proof}

	\begin{remark}\label{rem:conj-equiv-inclusion}
		The proof above shows, without using the hypothesis of Proposition~\ref{prop:reduction}, that
		$\mathcal{C}_{\bstar}\subseteq\F(\bstar)$ and that $\bstar\in\Gl$;
		combined with the $\F$-stability of $\Gl$ (Proposition~\ref{prop:Gl-F-stable}), this gives
		$\mathcal{C}_{\bstar}\subseteq\F(\bstar)\subseteq\Gl$ unconditionally.
		Hence Conjecture~\ref{conj:main1-all-types} is equivalent to the single inclusion $\Gl\subseteq\mathcal{C}_{\bstar}$.
	\end{remark}

	The following theorem is the main result of this paper.
	
	\begin{theorem}[$=$ Theorem 1]\label{thm:main1}
		If $\g$ is of type $A_{n-1}$,
		then $\F(\mathcal{C}_{\bstar})\subseteq \mathcal{C}_{\bstar}$.
		Therefore, Conjecture~\ref{conj:main1-all-types} is true for type $A_{n-1}$.
	\end{theorem}
	
	We will prove Theorem~\ref{thm:main1} in \S\ref{sec:proof-thm1}.
	
	\newcommand{\SSYT}{\operatorname{SSYT}}
	\newcommand{\Par}{\mathrm{Par}}
	
	%%%%%%%%%%%%%%%%%%%%%%%%%%%%%%%%%%%%%%%%%%%% 
	
	\section{Plactic monoid}\label{sec:plactic}
	In this section, 
	we assume that $\g$ is of type $A_{n-1}$, that is,
	$\g \cong \mathfrak{sl}(n)$.
	Define $\epsilon_{i} \in \mathfrak{h}^{\ast}$, $1\le i\le n$,
	by $\epsilon_{i} (\alpha_{j}^{\vee}):=\delta_{ij}-\delta_{i-1,j}$ for
	$j\in I=\{1,2,\dots,n-1\}$;
	note that
	$\alpha_{i} = \epsilon_{i} - \epsilon_{i+1}$ for $i \in I$,
	$\varpi_{i} = \epsilon_{1} + \epsilon_{2} + \cdots + \epsilon_{i}$ for $i \in I$,
	and $\epsilon_{1} + \epsilon_{2} + \cdots + \epsilon_{n} = 0$.
	Recall (see, for example, \cite{bumpschilling}) that the crystal basis of the vector
	representation $V(\varpi_1)$ of $U_q(\g)=U_q(\mathfrak{sl}(n))$ is given by
	$\mathbf{B}=\B(\varpi_1)=\{\boxed{k}\mid k=1,\dots,n\}$, equipped with 
	\[
	\wt(\boxed{k})=\epsilon_k,\qquad
	\varepsilon_i(\boxed{k})=\delta_{k,i+1},\qquad
	\varphi_i(\boxed{k})=\delta_{k,i},
	\]
	\[
	\tf_i(\boxed{k})=\begin{cases} \boxed{i+1} &\text{if } k=i,\\ 0& \text{otherwise,}\end{cases}
	\qquad
	\te_i(\boxed{k})=\begin{cases} \boxed{i} &\text{if }  k=i+1,\\ 0 &\text{otherwise.}\end{cases}
	\]
	
	We denote by $\Par_{\le n}$ the set of partitions
	$\lambda=(\lambda_1\ge\cdots\ge\lambda_n\ge0)$ of length at most $n$
	(recall that the length of $\lambda$ is defined to be
	$\max(\{1\le t\le n\mid\lambda_t>0\}\sqcup\{0\})$).
	For $\lambda\in\Par_{\le n}$,
	we denote by $\SSYT(\lambda)$ the set of
	semistandard Young tableaux of shape $\lambda$ over $\{1,2,\dots,n\}$, and set
	$\shape(T):=\lambda$ for $T\in\SSYT(\lambda)$.
	Also, we set $\iota(\lambda):=\sum_{k=1}^{n}\lambda_k\epsilon_k\in P^+$ and $|\lambda|:=\sum_{k=1}^{n}\lambda_k$.
	Notice that the map $\Par_{\le n}\rightarrow P^+$,
	$\lambda\mapsto \iota(\lambda)$, is surjective,
	but not injective since $\epsilon_{1}+\cdots+\epsilon_{n}=0$;
	recall that $P^+$ can be identified with the set of partitions of length less than $n$.
	
	\begin{remark}\label{rem:dominance-partition}
		For $\mu=(\mu_1\ge\cdots\ge\mu_n\ge0),\,\nu=(\nu_1\ge\cdots\ge\nu_n\ge0)\in\Par_{\le n}$,
		we write $\mu\dom\nu$ if $\sum_{k=1}^{j}\mu_k\ge\sum_{k=1}^{j}\nu_k$
		for all $1\le j\le n$.
		It is well-known that for $\mu,\nu\in\Par_{\le n}$ with $|\mu|=|\nu|$,
		$\mu\dom\nu$ if and only if $\iota(\mu)\ge\iota(\nu)$ (i.e., $\iota(\mu)-\iota(\nu)\in Q^+$).
	\end{remark}
	
	\begin{definition}\label{def:rw}
		For $T\in\SSYT(\lambda)$, the \emph{row reading word} $\rw(T)$ is the word
		obtained by reading the entries of $T$ row-by-row from bottom to top, and each
		row from left to right.
	\end{definition}
	
	For a word $u=u_1\cdots u_N$ over
	$\{1,2,\dots,n\}$, set
	\[
	\mathbf{b}(u):=\boxed{u_N}\otimes\cdots\otimes\boxed{u_1}\in\mathbf{B}^{\otimes N}.
	\]
	\begin{remark}
		For $T\in\SSYT(\lambda)$, 
		the element $\mathbf{b}(\rw(T))$ is the image of $T$ under the \emph{Middle-Eastern reading};
		see, for example, \cite[Definition 7.3.4(2)]{hongkong}.
	\end{remark}
	
	\begin{theorem}[{\cite[\S 3.1]{bumpschilling}}; see also {\cite[Chapter 7]{hongkong}}]\label{thm:rw-iso}
		Let $\lambda\in\Par_{\le n}$.
		The map
		\[ \Psi_\lambda:\SSYT(\lambda)\longrightarrow\mathbf B^{\otimes|\lambda|},\qquad
		T\mapsto\mathbf{b}(\rw(T)),
		\]
		is injective, and its image is a connected component of $\mathbf B^{\otimes|\lambda|}$
		isomorphic to $\B(\iota(\lambda))$.
	\end{theorem}
	
	By the theorem above,
	we can endow $\SSYT(\lambda)$ with the crystal structure via $\Psi_\lambda$ in such a way that
	$\SSYT(\lambda)\cong\B(\iota(\lambda))$ as crystals.
	
	\begin{definition}\label{def:unmatched}
		Fix $i\in I=\{1,2,\dots,n-1\}$.
		Let $u=u_1\cdots u_N$ be a word over $\{1,2,\dots,n\}$.
		Delete from $u$ every letter other than $i$ and $i+1$, and then delete
		adjacent pairs $(i+1,i)$ repeatedly until no such pair remains.
		A position
		$k\in\{1,2,\dots,N\}$ with $u_k=i$ is an \emph{unmatched} $i$ in $u$ if it survives this
		procedure.
	\end{definition}
	
	\begin{proposition}[{\cite{bumpschilling}}]\label{prop:signature-rule}
		For $u=u_1\cdots u_N$ a word over $\{1,2,\dots,n\}$, we have
		$$\tf_i\bigl(\mathbf{b}(u)\bigr)=\boxed{u_N}\otimes\cdots\otimes \tf_i(\;\boxed{u_j}\;)\otimes\cdots\otimes \boxed{u_1},$$
		where the position $j$ is the right-most unmatched $i$ in the word $u$.
		If $u$ has no unmatched $i$, then $\tf_i(\mathbf{b}(u))=0$.
	\end{proposition}
	
	\begin{definition}
		Let $\B_1,\B_2$ be crystals, and $b_i\in\B_i$ for $i=1,2$;
		recall that $C(b_i)$ denotes the connected component of $\B_i$ containing $b_i$.
		If there exists a (unique) isomorphism  $C(b_1)\xrightarrow{\ \sim\ }C(b_2)$ of crystals which maps $b_1$ to $b_2$,
		then we say that $b_1$ and $b_2$ are \emph{plactically equivalent},
		and write $b_1\equiv b_2$.
	\end{definition}
	\begin{remark}\label{rem:T-equiv-boxT}
		\begin{enumerate}
			\item[(i)] If $b_1\equiv b_2$ and $c_1\equiv c_2$,
			then $b_1\otimes c_1\equiv b_2\otimes c_2$.
			\item[(ii)] 
			By Theorem~\ref{thm:rw-iso}, 
			we have $T\equiv \mathbf{b}(\rw(T))$
			for $T\in\SSYT(\lambda)$.
		\end{enumerate}
	\end{remark}
	
	For $T\in\SSYT(\lambda)$ and a letter $i\in\{1,2,\dots,n\}$,
	we denote by $T\leftarrow i$
	the semistandard Young tableau obtained from $T$ by the \emph{Schensted insertion} of
	$i$ (see, for example, \cite[\S 7]{bumpschilling});
	for a word $w=w_1\cdots w_N$, we define
	\[
	\bigl(T\leftarrow w\bigr):=\bigl(\cdots(T\leftarrow w_1)\cdots\leftarrow w_N\bigr),
	\qquad
	P(w):=\bigl(\varnothing\leftarrow w\bigr);
	\]
	we call $P(w)$ the \emph{insertion tableau} of $w$.
	
	\begin{theorem}[{see, for example, \cite[\S 8.3]{bumpschilling}}]\label{thm:plactic-theory}
		Let $T\in\SSYT(\lambda)$.
		For $i\in\{1,2,\dots,n\}$,
		we have $\bigl(T\leftarrow i\bigr)\equiv \boxed{i}\otimes T.$
	\end{theorem}
	\begin{remark}\label{rem:word-component}
		Let $u=u_1\cdots u_N$ be an arbitrary word over $\{1,2,\dots,n\}$ with $N\ge1$.
		Then,
		\begin{align*}
		P(u)&=\bigl(\varnothing\leftarrow u\bigr)
		=\bigl(\cdots(\varnothing\leftarrow u_1)\cdots\leftarrow u_N\bigr)\\
		&\equiv \boxed{u_N}\otimes \bigl(\cdots(\varnothing\leftarrow u_1)\cdots\leftarrow u_{N-1}\bigr)\\
		&\equiv \boxed{u_N}\otimes \boxed{u_{N-1}}\otimes \bigl(\cdots(\varnothing\leftarrow u_1)\cdots\leftarrow u_{N-2}\bigr)\\
		&\equiv \cdots \equiv \boxed{u_N}\otimes \boxed{u_{N-1}}\otimes \cdots\otimes \boxed{u_1}=\mathbf{b}(u);
		\end{align*} 
		if we set $\mu:=\shape(P(u))$,
		then the element $\mathbf{b}(u)$ lies in a connected component of $\mathbf{B}^{\otimes N}$
		isomorphic to $\B(\iota(\mu))$.
	\end{remark}
	
	   For words $u=u_1\cdots u_N$ and $w=w_1\cdots w_M$,
	   we define $uw$ to be the concatenation $u_1\cdots u_N w_1\cdots w_M$.
	
	\begin{definition}[plactic product]\label{def:plactic-product}
		For $\mu,\nu\in\Par_{\le n}$, $T\in\SSYT(\mu)$ and $U\in\SSYT(\nu)$,
		we set
		\[
		T\cdot U \;:=\; P\bigl(\rw(T)\,\rw(U)\bigr) \;=\; \bigl(T\leftarrow \rw(U)\bigr).
		\]
	\end{definition}
	
	\begin{remark}\label{rem:plactic-tensor}
		For $\mu,\nu\in\Par_{\le n}$, $T\in\SSYT(\mu)$ and $U\in\SSYT(\nu)$, we have
		\[
		T\cdot U=P\bigl(\rw(T)\,\rw(U)\bigr)\equiv\mathbf{b}\bigl(\rw(T)\,\rw(U)\bigr)
		=\mathbf{b}(\rw(U))\otimes\mathbf{b}(\rw(T))\equiv U\otimes T,
		\]
		where the first $\equiv$ follows from Remark~\ref{rem:word-component},
		and the second $\equiv$ from Remark~\ref{rem:T-equiv-boxT}.
	\end{remark}
	
	\section{Greene invariants}\label{sec:greene}
	
	\newcommand{\Pos}{\operatorname{Pos}}
	
	\begin{definition}\label{def:greene-inv.}
		Let $u=u_1\cdots u_N$ be a word over $\mathbb{Z}_{>0}$,
		and $\Pos(u):=\{1,2,\dots,N\}$ the set of positions of $u$.
		\begin{enumerate}
			\item[(i)] A \emph{weakly increasing subsequence} (resp., \emph{strictly decreasing subsequence}) of $u$ is a subset
			$\{q_1,\dots,q_r\}$ of $\Pos(u)$ with $r\ge0$ and $q_1<\cdots<q_r$ satisfying the condition that 
			$$u_{q_1}\le\dots\le u_{q_r}\quad
			(\text{ resp., }u_{q_1}>\dots> u_{q_r})\;;
			$$
			note that the empty set is also a weakly increasing subsequence (resp., strictly decreasing subsequence) of $u$.
			\item[(ii)]
			For $k\ge1$, 
			the Greene invariant $G_k(u)$ of $u$ is defined to be
			\[
			G_k(u):=\max_{C_1,\dots,C_k}\ \sum_{t=1}^k|C_t|,
			\]
			where $C_1,\dots,C_k$ range over all pairwise disjoint weakly increasing subsequences of $u$. 
		\end{enumerate}
	\end{definition}
	
	\begin{theorem}[Greene {\cite{greene1974}}]\label{thm:greene}
		Let $u$ be a word over $\{1,2,\dots,n\}$, and write $\mu:=\shape(P(u))\in\Par_{\le n}$ as $\mu=(\mu_1\ge\cdots\ge \mu_n\ge 0)$.
		For $k\ge1$, we have
		$G_k(u)=\mu_1+\cdots+\mu_k,$
		where we set $\mu_t:=0$ for $t\ge n+1.$
	\end{theorem}
	
	\newcommand{\Ch}{\operatorname{Ch}}
	\newcommand{\ACh}{\operatorname{ACh}}
	
	\begin{definition}
		Let $(P,\le)$ be a poset.
		A \emph{chain} (resp., \emph{antichain}) is, by definition, a subset $Q$ of $P$ satisfying the condition that $x$ and $y$ are comparable (resp., not comparable) for every $x,y \in Q$ with $x\ne y$;
		note that a singleton (i.e., a subset $Q$ of $P$ such that $|Q|=1$) and the empty set are chains and antichains.
		Denote by $\Ch(P,\le)$ (resp., $\ACh(P,\le)$) the set of chains (resp., antichains) of $(P,\le)$.
		
	\end{definition}
	
	Let $(P,\le)$ be a finite poset.
	Set
	$$\mathcal{O}(P):=\Bigl\{\mathbf{C}=(C_1,\dots,C_k)\;\Bigm|\;k\ge1,\;C_i\in\Ch(P,\le)\ (i=1,\dots,k),\;\bigcup_{i=1}^k C_i=P\Bigr\},$$
	and $L(\mathbf{C}):=k$ for $\mathbf{C}=(C_1,\dots,C_k)\in \mathcal{O}(P)$.
	
	\begin{theorem}[Dilworth {\cite{dilworth1950}}]\label{thm:dilworth}
		Let $(P,\le)$ be a nonempty finite poset.
		We have 
		\[
		\max_{A\in \ACh(P,\le)}|A|=\min_{\mathbf{C}\in \mathcal{O}(P)}L(\mathbf{C}).
		\]
	\end{theorem}
	
	Fix a word $u=u_1\cdots u_N$ over $\mathbb{Z}_{>0}$.
	We define the finite poset $(\operatorname{Pos}(u),\preceq_u):=(\{1,2,\dots,N\},\preceq_u)$ of positions as follows:
	for $p,q\in \operatorname{Pos}(u)$, 
	$p\preceq_u q$ if $p\le q$ and $u_p\le u_q$.
	We note that a chain of $\operatorname{Pos}(u)$ is a weakly increasing subsequence of $u$,
	and an antichain of $\operatorname{Pos}(u)$ is a strictly decreasing subsequence of $u$
	(Definition~\ref{def:greene-inv.}).
	
	Let $S\subseteq\operatorname{Pos}(u)$,
	and let $(S,\preceq_u)$ be the subposet induced by $(\operatorname{Pos}(u),\preceq_u)$.
	We set 
	$$\ell(S,\preceq_u):=\max_{A\in\ACh(S,\preceq_u)}|A|;$$
	note that 
	$\ell(\varnothing,\preceq_u)=0$ since $\ACh(\varnothing,\preceq_u)=\{\varnothing\}$.
	
	\begin{lemma}\label{lem:tworow}
		Let $u=u_1\cdots u_N$ be a word over $\mathbb{Z}_{>0}$.
		For $k\ge1$, we have
		$G_k(u)=\max\{|S|\mid S\subseteq\operatorname{Pos}(u),\ \ell(S,\preceq_u)\le k\}.$
	\end{lemma}

	\begin{proof}
		Fix $k\ge1$.
		First, we show that $G_k(u)\ge|S|$ for all $S\subseteq\operatorname{Pos}(u)$ with
		$\ell(S,\preceq_u)\le k$.
		If $S=\varnothing$, then the inequality is obvious since $G_k(u)\ge0=|S|$.
		Assume that $S\ne\varnothing$, and set $t:=\ell(S,\preceq_u)$;
		note that $1\le t\le k$.
		By Theorem~\ref{thm:dilworth},
		there exist $C_1,\dots,C_t\in \Ch(S,\preceq_u)$ such that $S=\bigcup_{i=1}^t C_i$.
		For $1\le i\le k$, we set
		$$
		C_i':=\left\{
		\begin{aligned}
			&C_i\setminus\bigcup_{j>i}C_j\;
			\quad &\text{for}\quad 1\le i\le t,  \\
			&\varnothing\;
			\quad  &\text{for}\quad t<i\le k.
		\end{aligned}
		\right.
		$$
		The sets $C_1',\dots,C_k'$ are pairwise disjoint chains of $\operatorname{Pos}(u)$ such that $\bigsqcup_{i=1}^k{C_i'}=\bigcup_{i=1}^t C_i=S$.
		Thus, by Definition~\ref{def:greene-inv.},
		\[
		G_k(u)\ge\sum_{i=1}^k|C_i'|=|S|.
		\]
		
		Next, we show that there exists $S\subseteq\operatorname{Pos}(u)$ such that $\ell(S,\preceq_u)\le k$ and
		$|S|=G_k(u)$.
		By the definition, there exist pairwise
		disjoint chains $C_1,\dots,C_k$ of $\operatorname{Pos}(u)$ such that $G_k(u)=\sum_{i=1}^k|C_i|$.
		Set $S:=\bigsqcup_{i=1}^kC_i$; 
		note that $|S|=G_k(u)$.
		Let us show that $\ell(S,\preceq_u)\le k$.
		Let $A\in \ACh(S,\preceq_u)$;
		note that $|A|=\sum_{i=1}^k |A\cap C_i|$.
		Let $1\le i\le k$.
		Because two elements in the chain $C_i$ are comparable, 
		and because $A$ is an antichain,
		we see $|A\cap C_i|\le 1$.
		Thus we obtain $|A|\le k$,
		which gives $\ell(S,\preceq_u)\le k$.
	\end{proof}
	
	The following theorem will be used in the proof of our main theorem (Theorem~\ref{thm:main1}).
	
	\begin{theorem}\label{thm:master}
		
		Fix $i\in I=\{1,2,\dots,n-1\}$.
		Let $u=u_1\cdots u_{p^\ast}\cdots u_N$ be a word over $\{1,2,\dots,n\}$,
		with $p^\ast$ an unmatched $i$ in $u$ (Definition~\ref{def:unmatched}).
		Let $u':=u_1\cdots u_{p^\ast}'\cdots u_N$
		be the word over $\{1,2,\dots,n\}$ obtained from $u$ by replacing $u_{p^\ast}=i$ with $u_{p^\ast}':=i+1$.
		Then, for $k\ge1$, we have $G_k(u)\ge G_k(u').$
	\end{theorem}
	
	\begin{proof}
		Fix $k\ge 1$. 
		By the definition,
		there are pairwise disjoint chains
		$C_1',\dots,C_k'$ of $\operatorname{Pos}(u')$ such that $G_k(u')=\sum_{t=1}^k|C_t'|$.
		Set $S':=\bigsqcup_{t=1}^k C_t'$;
		note that $|S'|=G_k(u')$.
		Also, we have $\ell(S',\preceq_{u'})\le k$, since $|A\cap C_t'|\le1$ for $A\in\ACh(S',\preceq_{u'})$ and $1\le t\le k$; see the proof of Lemma~\ref{lem:tworow}.
		
		In order to prove $G_k(u)\ge G_k(u')$,
		it suffices to construct a subset $S$ of the poset $\operatorname{Pos}(u)$ (from $S'$)
		such that $|S|=|S'|$ and $\ell(S,\preceq_u)\le k$;
		indeed, if this is done, then
		it follows from Lemma~\ref{lem:tworow} that
		\[
		G_k(u)\ge |S| = |S'|=G_k(u').
		\]
		If $p^\ast\notin S'$, 
		then we set $S:=S'$;
		it is obvious that $|S|=|S'|$.
		Also, we deduce that $\ell(S,\preceq_u)=\ell(S',\preceq_{u'})\le k$, 
		since the subword of $u$ corresponding to $S$ is equal to the subword of $u'$ corresponding to $S'$.
		Assume that $p^\ast\in S'$.
		Note that if $A\subseteq\Pos(u)=\{1,2,\dots,N\}$ is an antichain of $(\Pos(u),\preceq_u)$ with
		$p^\ast\notin A$, then $u_x=u'_x$ for all $x\in A$, and hence $A$ is also an
		antichain of $(\Pos(u'),\preceq_{u'})=(\{1,2,\dots,N\},\preceq_{u'})$; we refer to this fact as $(\dagger)$.
		
		\paragraph{\bf Step 1.} We construct $S$ as follows.
		Set $Q_{S'}:=\{q\in S'\mid q<p^\ast\text{ and }u_q=i+1\}.$
		Since $p^\ast$ is an unmatched $i$,
		we deduce that for each $q\in Q_{S'}$,
		there exists a unique $y(q)\in\operatorname{Pos}(u)$ with $q<y(q)<p^\ast$ such that $u_{y(q)}=i$ and the pair $(u_q,u_{y(q)})=(i+1,i)$ is deleted in the procedure of Definition~\ref{def:unmatched}.
		Thus we obtain an injective map $y:Q_{S'}\to \{1,2,\dots,p^\ast-1\},\ q\mapsto y(q)$.
		For $q\in \operatorname{Pos}(u)=\{1,2,\dots,N\}$, we set
		\begin{align*}
			\operatorname{Left}(q)_{>i+1}&:=\left\{p\in S'\mid p<q\text{ and }u_p>i+1\right\}, \\
			\operatorname{Right}(q)_{<i\;}&:=\left\{p\in S'\mid p>q\text{ and }u_p<i\right\}.
		\end{align*}
		Also, we set
		\begin{align*}
			D&:=\Bigl\{q\in Q_{S'}\mid\ell(\operatorname{Left}(q)_{>i+1},\preceq_u)+\ell(\operatorname{Right}(p^\ast)_{<i},\preceq_u)\ge k-1\Bigr\},\\
			Y&:=\{y(q)\mid q\in D\},\\
			S&:=(S'\setminus D)\cup Y.
		\end{align*}
		\paragraph{\bf Step 2.} We prove that $|S|=|S'|$.
		Notice that $|D|=|Y|$ since $y$ is injective,
		and that $D\subseteq Q_{S'}\subseteq S'$.
		Hence it suffices to verify that
		$Y\cap S'=\varnothing$.
		Suppose, for a contradiction, that
		$Y\cap S'\ne\varnothing$.
		Let $y\in Y\cap S'$, 
		and let $q\in D$ be such that $y=y(q)$.
		Set $a:=\ell(\operatorname{Left}(q)_{>i+1},\preceq_u)$ and $b:=\ell(\operatorname{Right}(p^\ast)_{<i},\preceq_u).$
		Let $\{r_1,\dots,r_a\}\in\ACh(\operatorname{Left}(q)_{>i+1},\preceq_u)$ with $r_1<\dots<r_a$,
		and
		$\{s_1,\dots,s_b\}\in\ACh(\operatorname{Right}(p^\ast)_{<i},\preceq_u)$ with $s_1<\dots<s_b$;
		recall that $u_{r_1}>\cdots>u_{r_a}$ and $u_{s_1}>\cdots>u_{s_b}$.
		Since $q\in D$, we have $a+b\ge k-1$;
		we take $a'\le a$ and $b'\le b$ in such a way that $a'+b'=k-1$.
		Then we see that
		\[
		\left\{\begin{aligned}
			&r_{1}<\cdots<r_{a'}<q<y<s_1<\cdots<s_{b'},\\
			&u_{r_{1}}>\cdots>u_{r_{a'}}>\underset{=i+1}{\underbrace{u_q}}>\underset{=i}{\underbrace{u_y}}>u_{s_1}>\cdots>u_{s_{b'}}.\end{aligned}\right.
		\]
		Hence, $A:=\{r_{1},\dots,r_{a'},\,q,\,y,\,s_1,\dots,s_{b'}\}\in\ACh(\Pos(u),\preceq_u)$;
		observe that $|A|=k+1$ since $a'+b'=k-1$.
		It follows that $p^\ast\notin A$ since $y\in A$, $y<p^\ast$, and $u_y=i=u_{p^\ast}$ (in particular,
		$y$ and $p^\ast$ are comparable in $(\Pos(u),\preceq_u)$), 
		and since
		$A$ is an antichain.
		Therefore, $A\in\ACh(\operatorname{Pos}(u'),\preceq_{u'})$
		by $(\dagger)$.
		Moreover, we see that $A\subseteq S'$,
		since $\operatorname{Left}(q)_{>i+1},\operatorname{Right}(p^\ast)_{<i}$ and $D$ are included in $S'$ by the definitions and since $y\in S'$ by assumption.
		This implies that $\ell(S',\preceq_{u'})\ge|A|=k+1$, which contradicts $\ell(S',\preceq_{u'})\le k$.
		\paragraph{\bf Step 3.} 
		We show the following claim ($\star$):
		if $A=\{x_1, x_2, \dots,x_{m}\}\subseteq S'$ is an antichain of $\Pos(u)$ with 
		$x_1<x_2<\cdots<x_m$ and $m\ge k+1$, 
		then $p^\ast=x_j\in A$ for some $j=2,\dots,m$,
		and $x_{j-1}\in D$.
		Indeed,
		if $p^\ast\notin A$,
		then $A$ is an antichain of $(\Pos(u'),\preceq_{u'})$ included in $S'$ by
		$(\dagger)$,
		which contradicts $\ell(S',\preceq_{u'})\le k$.
		Thus we get $p^\ast\in A$.
		Let $j$ be such that $x_j=p^\ast$.
		Assume that $j=1$.
		Since $u'_{p^\ast}=i+1>i=u_{p^\ast}=u_{x_1}>u_{x_2}$,
		it follows that $A=\{p^\ast=x_1,x_2,\dots,x_{m}\}\in \ACh(S',\preceq_{u'})$,
		which contradicts $\ell(S',\preceq_{u'})\le k$.
		Thus we get $j\ge2$; note that $u_{x_{j-1}}>u_{x_j}=i$, 
		and hence $u_{x_{j-1}}\ge i+1$.
		If $u_{x_{j-1}}\ge i+2$, then $u'_{x_{j-1}}=u_{x_{j-1}}>u'_{p^\ast}=i+1$.
		As above, we see that $A\in\ACh(S',\preceq_{u'})$,
		which contradicts $\ell(S',\preceq_{u'})\le k$.
		Hence we conclude that $u_{x_{j-1}}=i+1$.
		Since $x_{j-1}\in S'$ and $x_{j-1}<p^\ast$, we get $x_{j-1}\in Q_{S'}$.
		Moreover, since $u_{x_1}>\cdots>u_{x_{j-2}}>u_{x_{j-1}}=i+1$,
		it follows that $\{x_1,\dots,x_{j-2}\}\in\ACh(\operatorname{Left}(x_{j-1})_{>i+1},\preceq_u)$,
		and hence $\ell(\operatorname{Left}(x_{j-1})_{>i+1},\preceq_u)\ge j-2$.
		Since $p^\ast=x_j<x_{j+1}$,
		we see that $\{x_{j+1},\dots,x_{m}\}\in\ACh(\operatorname{Right}(p^\ast)_{<i},\preceq_u)$,
		and hence $\ell(\operatorname{Right}(p^\ast)_{<i},\preceq_u)\ge m-j\ge k+1-j$.
		Therefore we obtain
		$$
		\ell(\operatorname{Left}(x_{j-1})_{>i+1},\preceq_u)+
		\ell(\operatorname{Right}(p^\ast)_{<i},\preceq_u)\ge (j-2)+(k+1-j)=k-1,
		$$
		which implies that $x_{j-1}\in D$.
		This proves claim ($\star$).
		\paragraph{\bf Step 4.}     
		We show that $\ell(S,\preceq_u)\le k$.
		Since $p^\ast\notin Q_{S'}$ and $D\subseteq Q_{S'}$, it follows that
		$p^\ast\in S'\setminus D$.
		Suppose, for a contradiction, that $\ell(S,\preceq_u)> k$.
		There exists $A\in\ACh(S,\preceq_u)$ with $|A|=k+1$; write $A$ as $A=\{x_1,x_2,\dots,x_{k+1}\}$ with $x_1<x_2<\cdots<x_{k+1}$.
		Here we recall that $S=(S'\setminus D)\cup Y$.
		We claim that $A\cap Y\ne\varnothing$.
		Suppose, for a contradiction, that $A\cap Y=\varnothing$.
		Then we have $A\in \ACh(S'\setminus D,\preceq_u)$ $\subseteq\ACh(S',\preceq_u)$. It follows from claim ($\star$) in Step 3 that $p^\ast=x_j$ and $x_{j-1}\in D$ for some $j$.
		However, this contradicts $x_{j-1}\in A\subseteq S'\setminus D$.
		Thus, $A\cap Y\ne\varnothing$.
		Since $u_y=i$ for all $y\in Y$,
		and since $A\in\ACh(S,\preceq_u)$,
		we have $|A\cap Y|=1$.
		Let $1\le r\le k+1$ be such that 
		$A\cap Y=\{x_r\}$.
		Set $y:=x_r$, and let $q\in D$ be such that $y=y(q)$.
		Since $y<p^\ast$ and $u_y=i=u_{p^\ast}$, 
		and since $y\in A$ and $A\in\ACh(S,\preceq_u)$,
		we get $p^\ast\notin A$.
		Let $B=\{z_1,\dots,z_a\}\in\ACh(\operatorname{Left}(q)_{>i+1},\preceq_u)$ with $a=\ell(\operatorname{Left}(q)_{>i+1},\preceq_u)\ge0$ and
		$C=\{w_1,\dots,w_c\}\in\ACh(\operatorname{Right}(p^\ast)_{<i},\preceq_u)$ with $c=\ell(\operatorname{Right}(p^\ast)_{<i},\preceq_u)\ge0$;
		note that $a+c\ge k-1$ since $q\in D$.
		The set $B':=B\sqcup\{q\}\sqcup\{x_{r+1},\dots,x_{k+1}\}\in\ACh(S',\preceq_u)$ since
		\[
		\left\{\begin{aligned}
			&z_1<\cdots<z_a<q<y=y(q)=x_r<x_{r+1}<\cdots<x_{k+1},\\
			&u_{z_1}>\cdots>u_{z_a}>i+1=u_q>i=u_y=u_{x_r}>u_{x_{r+1}}>\cdots>u_{x_{k+1}};\end{aligned}\right.
		\]
		$B'$ does not contain $p^\ast$, since $u_x\ne i=u_{p^\ast}$ for all $x\in B'$.
		By claim ($\star$) in Step 3,
		we have $a+1+(k+1-r)\le k$, and hence
		$a\le r-2$; in particular, $r\ge2$.
		Moreover,
		the set $C':=\{x_1,\dots,x_{r-1}\}\sqcup\{p^\ast\}\sqcup C\in\ACh(S',\preceq_u)$ since
		\[
		\left\{\begin{aligned}
			&x_1<\cdots<x_{r-1}(<x_r=y=y(q))<p^\ast<w_1<\cdots<w_c,\\
			&u_{x_1}>\cdots>u_{x_{r-1}}>u_{x_r}=u_y=i=u_{p^\ast}>u_{w_1}>\cdots>u_{w_c}.\end{aligned}\right.
		\]
		By claim ($\star$), together with the fact that $x_{r-1}\in A\setminus\{y\}=A\setminus(A\cap Y)\subseteq S'\setminus D$,
		we deduce that $r+c\le k$.
		Therefore $a+c\le(r-2)+(k-r)=k-2$, which contradicts $q\in D$.
		Thus we obtain $\ell(S,\preceq_u)\le k$, as desired.
		This completes the proof of the theorem.
	\end{proof}
	
	\section{Proof of Theorem~\ref{thm:main1}}\label{sec:proof-thm1}
	
	Recall the notation from \S \ref{sec:plactic}.
	
	\begin{proposition}\label{prop:Tw-ge-fiTw}
		Let $T\in \SSYT(\lambda)$,
		and $u$ a word over $\{1,2,\dots,n\}$.
		Fix $i\in I=\{1,2,\dots,n-1\}$ such that $\tf_i(T)\ne0$.
		Then,
		$
		\shape\bigl(T\cdot P(u)\bigr)\dom \shape\bigl(\tf_i(T)\cdot P(u)\bigr).
		$
	\end{proposition}
	
	\begin{proof}
		Write $\rw(T)$ as $\rw(T)=x_1x_2\cdots x_N$.
		By Proposition~\ref{prop:signature-rule},
		there exists $1\le p\le N$ such that $x_p=i$ is an unmatched $i$ in $\rw(T)$,
		and $\rw(\tf_i(T))$ is obtained from $\rw(T)$ by replacing $x_p=i$ with $x_p'=i+1$.
		Since whether a letter $i$ is unmatched depends only on the letters to its left
		(Definition~\ref{def:unmatched}),
		we deduce that this $x_p=i$ is an unmatched $i$ also in $\rw(T)v$, where $v:=\rw(P(u))$.
		Hence it follows from Theorem~\ref{thm:master} that $G_k\bigl(\rw(T)v\bigr)\ge G_k\bigl(\rw(\tf_i(T))v\bigr)$ for all $k\ge1$;
		notice that $|\shape\bigl(P(\rw(T)v)\bigr)|=|\shape\bigl(P(\rw(\tf_i(T))v)\bigr)|$.
		Therefore, by Definition~\ref{def:plactic-product} and Theorem~\ref{thm:greene}, together with the definition of $\dom$, we obtain
		$$\shape\bigl(T\cdot P(u)\bigr)=
		\shape\bigl(P(\rw(T)v)\bigr)\dom 
		\shape\bigl(P(\rw(\tf_i(T))v)\bigr)=\shape\bigl(\tf_i(T)\cdot P(u)\bigr),$$
		as desired.
	\end{proof}
	
	\begin{proposition}\label{prop:w0Tgew0fiT}
		Let $T\in \SSYT(\lambda)$ and $i\in I$ be such that $\tf_i(T)\ne0$.
		Then,
		\[
		\wt\bigl(\chke_{w_\circ}(T)\bigr)\ge\wt\bigl(\chke_{w_\circ}(\tf_i(T))\bigr).
		\]
	\end{proposition}
	\begin{proof}
		Set $\varrho:=(n-1,n-2,\dots,2,1,0)\in\Par_{\le n}$; note that $\varrho_n=0$ and $\iota(\varrho)=\rho=\sum_{i\in I}\varpi_i$.
		Let $T_\rho\in\SSYT(\varrho)$ be the highest tableau (whose entries in the $j$-th row are all equal to $j$),
		which corresponds to the highest element $b_\rho\in\B(\rho)$.
		It suffices to show
		\begin{align}\label{delete-chkw-operator}
			\wt\bigl(T_\rho\otimes\chke_{w_\circ}(T)\bigr)\ge\wt\bigl(T_\rho\otimes\chke_{w_\circ}(\tf_i(T))\bigr),
		\end{align}
		or equivalently,
		$\wt\bigl(\te_{w_\circ}^{\max}(T_\rho\otimes T)\bigr)\ge\wt\bigl(\te_{w_\circ}^{\max}(T_\rho\otimes\tf_i(T))\bigr)$;
		see Proposition~\ref{prop:eimax-to-chkw}.
		By Remark~\ref{rem:plactic-tensor},
		we have $T_\rho\otimes T\equiv T\cdot T_\rho$
		and $T_\rho\otimes\tf_i(T)\equiv\tf_i(T)\cdot T_\rho$.
		Since the Kashiwara operators preserve the shape of a semistandard Young tableau, and since
		$\te_{w_\circ}^{\max}(T\cdot T_\rho)$ is the highest element of the connected
		component containing it, we have
		$
		\wt\bigl(\te_{w_\circ}^{\max}(T\cdot T_\rho)\bigr)=\iota\bigl(\shape(T\cdot T_\rho)\bigr).
		$
		Similarly,
		we have $\wt\bigl(\te_{w_\circ}^{\max}(T_\rho\otimes\tf_i(T))\bigr)=\iota\bigl(\shape(\tf_i(T)\cdot T_\rho)\bigr)$.
		Note that $|\shape\bigl(T\cdot T_\rho\bigr)|=|\shape\bigl(\tf_i(T)\cdot T_\rho\bigr)|=|\lambda|+|\varrho|$.
		Thus we see by Remark~\ref{rem:dominance-partition} that inequality \eqref{delete-chkw-operator} is equivalent to
		$
		\shape\bigl(T\cdot T_\rho\bigr)\dom\shape\bigl(\tf_i(T)\cdot T_\rho\bigr),
		$
		which follows from Proposition \ref{prop:Tw-ge-fiTw}.
	\end{proof}
	
	\begin{proof}[Proof of Theorem~\ref{thm:main1}]
		
		We now prove Theorem~\ref{thm:main1}.
		We deduce from Propositions~\ref{prop:eimax-to-chkw} and  \ref{prop:chkwmax-check-stable} that
		\[
		\te_i^{\max}(b_\rho\otimes{b_\star})=b_\rho\otimes\chke_{i}(b_\star)=b_\rho\otimes{b_\star}
		\qquad\text{for all }i\in I,
		\]
		which implies that $b_\rho\otimes{b_\star}$ is a highest element of $\B(\rho)\otimes\B(\lambda)$.
		We claim that 
		\begin{align}\label{eq:wt-min}
			\wt\bigl(\te_{w_\circ}^{\max}(b_\rho\otimes b)\bigr)\ge\wt(b_\rho\otimes{b_\star})
		\end{align}
		for all $b\in\Bl$; 
		recall that $\te_{w_\circ}^{\max}(b_\rho\otimes b)=b_\rho\otimes\chke_{w_\circ}(b)$ is the highest element of the connected component $C(b_\rho\otimes b)$ of $\B(\rho)\otimes\Bl$.
		Indeed,
		let $b\in\Bl$.
		Then, $b_{w_\circ\lambda}=\tf_{i_1}\cdots \tf_{i_k}(b)$ for some $i_1,\dots,i_k\in I$ with $k\ge 0$.
		We see by Proposition~\ref{prop:w0Tgew0fiT}, together with  Proposition~\ref{prop:eimax-to-chkw}, that
		\begin{align*}
			&\wt\bigl(\te_{w_\circ}^{\max}(b_\rho\otimes b)\bigr)
			=\wt\bigl(b_\rho\otimes\chke_{w_\circ}(b)\bigr)
			=\rho+\wt\bigl(\chke_{w_\circ}(b)\bigr)\\
			&\ge \rho+\wt\bigl(\chke_{w_\circ}(\tf_{i_k}b)\bigr)\ge\cdots\ge \rho+\wt\bigl(\chke_{w_\circ}(\tf_{i_{1}}\cdots\tf_{i_k}b)\bigr)\\
			&=\rho+\wt\bigl(\chke_{w_\circ}(b_{w_\circ\lambda})\bigr)=\wt(b_\rho\otimes b_\star).
		\end{align*}
		
		Here we deduce from 
		\cite{prv1967} (see also \cite[\S 3]{khare2012})
		that there exists a unique minimal element $\nu$ (with respect to the dominance order) in the set of all $\mu\in P^+$ such that the multiplicity of $V(\mu)$
		in $V(\rho)\otimes V(\lambda)$ is greater than 0;
		moreover, the multiplicity of $V(\nu)$ in $V(\rho)\otimes V(\lambda)$ is equal to 1.
		By this fact and inequality \eqref{eq:wt-min},
		we see that $\wt(b_\rho\otimes b_{\star})=\nu.$
		
		Now, let us show that $\F(\mathcal{C}_{b_\star})\subseteq\mathcal{C}_{b_\star}$.
		Let $b\in \mathcal{C}_{b_\star}$ and $i\in I$ be such that $\tf_i(b)\ne0$.
		We deduce again from Proposition~\ref{prop:w0Tgew0fiT} that
		\[
		\wt(b_\star)=\wt\bigl(\chke_{w_\circ}(b)\bigr)\ge\wt\bigl(\chke_{w_\circ}(\tf_i(b))\bigr),
		\]
		which implies that $\wt(b_\rho\otimes b_\star)\ge\wt\bigl(b_\rho\otimes\chke_{w_\circ}(\tf_i(b))\bigr)$.
		As seen above, we have $\wt\bigl(b_\rho\otimes\chke_{w_\circ}(\tf_i(b))\bigr)\ge\wt(b_\rho\otimes b_\star)$.
		Thus we obtain $\wt\bigl(b_\rho\otimes\chke_{w_\circ}(\tf_i(b))\bigr)=\wt(b_\rho\otimes b_\star)=\nu.$
		Since the multiplicity of $V(\nu)$ in $V(\rho)\otimes V(\lambda)$ is equal to $1$,
		and since both $b_\rho\otimes b_{\star}$ and $b_\rho\otimes\chke_{w_\circ}(\tf_i(b))=\te_{w_\circ}^{\max}(b_\rho\otimes \tf_i(b))$ are highest elements of weight $\nu$,
		we get $b_\rho\otimes\chke_{w_\circ}(\tf_i(b))=b_\rho\otimes b_\star,$
		and hence $b_\star=\chke_{w_\circ}(\tf_i(b))$.
		Thus we obtain $\tf_i(b)\in \mathcal{C}_{b_\star}$, which proves $\F(\mathcal{C}_{\bstar})\subseteq\mathcal{C}_{\bstar}$.
	\end{proof}

	\begin{remark}
		In the proof above, the type $A$ assumption is used only through Proposition~\ref{prop:w0Tgew0fiT};
		every other step is valid for $\g$ of arbitrary finite type.
		Hence Conjecture~\ref{conj:main1-all-types} holds for any finite type $\g$ for which an analogue of
		Proposition~\ref{prop:w0Tgew0fiT} is available.
		Moreover, by the Parthasarathy--Ranga Rao--Varadarajan theorem \cite{prv1967}
		(see also \cite[\S3]{khare2012}), the weight $\nu$ above is the highest weight $\overline{\rho+w_\circ\lambda}$
		of the PRV component of $V(\rho)\otimes V(\lambda)$, where $\overline{\mu}$ denotes the unique dominant weight
		in the $W$-orbit of $\mu\in P$; thus $\wt(\bstar)=\nu-\rho=\overline{\rho+w_\circ\lambda}-\rho$.
	\end{remark}
	%%%%%%%%%%%%%%%%%%%%%%%%%%%%%%%%%%%%%%%%%%%%%%%%%%%%%%%%%%%%%%%%
	\appendix
	\section{Further properties of the subsets \texorpdfstring{$\mathcal{C}_b$}{C\_b}}\label{sec:applications}
	%%%%%%%%%%%%%%%%%%%%%%%%%%%%%%%%%%%%%%%%%%%%%%%%%%%%%%%%%%%%%%%%
	
	We denote by $\le$ the Bruhat order on $W$.
	Fix $\lambda\in P^+.$
	Set $W_\lambda:=\{w\in W\mid w\lambda=\lambda\}$.
	For $w\in W$, denote by $\lfloor w\rfloor=\lfloor w\rfloor^\lambda$ the minimal-length coset representative of $wW_\lambda$,
	and set $W^\lambda:=\{\lfloor w\rfloor=\lfloor w\rfloor^\lambda\mid w\in W\}\subseteq W$.
	Recall that $\B_w(\lambda)$, $w\in W$,
	denotes the Demazure crystal associated with $w\in W$;
	note that $\B_w(\lambda)=\B_v(\lambda)$ if $wW_\lambda=vW_\lambda$.
	
	\begin{definition}[{see \cite{mason2009}
    \cite{assafgonzalez2025}}]\label{def:demazureatom}
		Let $\lambda\in P^+$.
		The Demazure atom corresponding to $w\in W^\lambda$ is defined to be $\mathcal{A}_w(\lambda):=\B_w(\lambda)\setminus\bigl(\bigcup_{v\in W^\lambda,v<w} \B_{v}(\lambda)\bigr)$.
	\end{definition}
	\begin{remark}\label{rem:atomblock}
		We have $\Bl=\bigsqcup_{v\in W^\lambda}\mathcal{A}_v(\lambda)$, and $\B_w(\lambda)=\bigsqcup_{v\in W^\lambda,\,v\le w}\mathcal{A}_v(\lambda)$ for $w\in W^\lambda$.
	\end{remark}
	
	\begin{proposition}[cf. Proposition~\ref{prop:unique-checked-rep}]\label{prop:decomp-into-classes}
		Let $X$ be the Demazure atom $\mathcal{A}_w(\lambda)$ for some $w\in W^\lambda$
		or an extremal subset $E\in\mathbb{E}_\lambda$. Then,
		$$X=\bigsqcup_{b\in X\cap\Tfix}\mathcal{C}_b.$$
	\end{proposition}
	\begin{proof}
		The proof is divided into four steps.
		\paragraph{\bf Step 1.}
		We show that
		\begin{align}\label{eq:string-minus-top-in-X}
			S^i(b)\setminus\{\te_i^{\max}(b)\}\subseteq X
			\quad\text{for all } b\in X \text{ and } i\in I \text{ such that } \te_i(b)\ne0.
		\end{align}
		Let $b\in X$ and $i\in I$ be such that $\te_i(b)\ne0$; note that $b\ne\te_i^{\max}(b)$.
		First, assume that $X=E\in\mathbb{E}_\lambda$.
		Since $b\in S^i(b)\cap E$ and $b\ne\te_i^{\max}(b)$,
		we have $S^i(b)\cap E=S^i(b)$ by \eqref{i-stringproperty}, which implies \eqref{eq:string-minus-top-in-X}.
		Next, assume that $X=\mathcal{A}_w(\lambda)$.
		Since $b\in\B_w(\lambda)$, $b\ne\te_i^{\max}(b)$, and $\B_w(\lambda)$ is an extremal subset,
		we have $S^i(b)\subseteq\B_w(\lambda)$ by \eqref{i-stringproperty}.
		Also, for each $v\in W^\lambda$ with $v<w$, we have $b\notin\B_v(\lambda)$,
		and hence $S^i(b)\cap\B_v(\lambda)\in\{\varnothing,\{\te_i^{\max}(b)\}\}$ by \eqref{i-stringproperty}.
		Therefore we obtain
		$S^i(b)\setminus\{\te_i^{\max}(b)\}\subseteq\B_w(\lambda)\setminus\bigcup_{v\in W^\lambda,\,v<w}\B_v(\lambda)=\mathcal{A}_w(\lambda)$.
		
		\paragraph{\bf Step 2.}
		We show that $\chke_i(X)\subseteq X$ for all $i\in I$;
		in particular, $\chke_{\wo}(b)\in X$ for all $b\in X$.
		Let $b\in X$ and $i\in I$.
		If $\te_i(b)=0$, then $\chke_i(b)=b\in X$ by Definition~\ref{def:chke}.
		If $\te_i(b)\ne0$, then $\chke_i(b)=\tf_i\te_i^{\max}(b)$ is an element of $S^i(b)$ distinct from $\te_i^{\max}(b)$,
		and hence $\chke_i(b)\in X$ by \eqref{eq:string-minus-top-in-X}.
		
		\paragraph{\bf Step 3.}
		We show that $\mathcal{C}_b\subseteq X$ for all $b\in X$.
		Let $b\in X$ and $c\in\mathcal{C}_b$.
		Take a reduced expression $\wo=s_{j_1}\cdots s_{j_N}$ of $\wo$.
		Set $c_{N+1}:=c$ and $c_k:=\chke_{j_k}(c_{k+1})$ for $k=N,N-1,\dots,1$;
		note that $c_1=\chke_{\wo}(c)=\chke_{\wo}(b)$.
		We show that $c_k\in X$ for $k=1,2,\dots,N+1$ by induction on $k$.
		If $k=1$, then $c_1=\chke_{\wo}(b)\in X$ by Step 2.
		Let $1\le k\le N$, and assume that $c_k\in X$.
		If $c_{k+1}=c_k$, then $c_{k+1}\in X$.
		Assume that $c_{k+1}\ne c_k=\chke_{j_k}(c_{k+1})$.
		Then $\te_{j_k}(c_{k+1})\ne0$ and $c_k=\tf_{j_k}\te_{j_k}^{\max}(c_{k+1})$ by Definition~\ref{def:chke};
		hence $S^{j_k}(c_k)=S^{j_k}(c_{k+1})$ and $\te_{j_k}(c_k)\ne0$.
		Since $c_{k+1}\ne\te_{j_k}^{\max}(c_{k+1})$,
		we deduce from \eqref{eq:string-minus-top-in-X}, applied to $c_k\in X$, that
		$c_{k+1}\in S^{j_k}(c_k)\setminus\{\te_{j_k}^{\max}(c_k)\}\subseteq X$.
		Thus we obtain $c=c_{N+1}\in X$.
		
		\paragraph{\bf Step 4.}
		Let $c\in X$,
		and set $d:=\chke_{\wo}(c)$;
		note that $d\in X\cap\Tfix$ by Step 2 and Proposition~\ref{prop:chkwmax-check-stable},
		and that $c\in\mathcal{C}_d$.
		Hence, $X\subseteq\bigsqcup_{b\in X\cap\Tfix}\mathcal{C}_b$;
		the reverse inclusion follows from Step 3.
		Thus we have proved the proposition.
	\end{proof}
	
	\begin{proposition}\label{prop:union-of-classes}
		Let $R\subseteq\Tfix$,
		and set $X:=\bigsqcup_{b\in R}\mathcal{C}_{b}$.
		Then, $X$ is an extremal subset if and only if $X$ is $\te_i$-stable for all $i\in I$.
	\end{proposition}
	\begin{proof}
		The implication ($\Rightarrow$) follows from Proposition~\ref{prop:ext-e-stable}.
		We prove the converse implication.
		Assume that $X$ is $\te_i$-stable for all $i\in I$.
		Let $b\in\Bl$ and $i\in I$ be such that $S^i(b)\cap X\ne\varnothing$.
		Since $X$ is $\te_i$-stable, we have $\te_i^{\max}(b)\in X$.
		Assume further that $S^i(b)\cap X\ne\{\te_i^{\max}(b)\}$,
		and take $c\in S^i(b)\cap X$ with $c\ne\te_i^{\max}(b)$.
		Let $d\in R$ be such that $c\in\mathcal{C}_d$.
		Since $c\in\mathcal{C}_d\cap S^i(b)$ and $c\ne\te_i^{\max}(b)$,
		we deduce from Proposition~\ref{prop:weak-string-property} that
		$S^i(b)\setminus\{\te_i^{\max}(b)\}\subseteq\mathcal{C}_d\cap S^i(b)\subseteq X$,
		and hence that $S^i(b)\cap X=S^i(b)$.
		Therefore,
		$S^i(b)\cap X\in\{\varnothing,\;S^i(b),\;\{\te_i^{\max}(b)\}\}$
		for all $b\in\Bl$ and $i\in I$,
		and hence $X$ is an extremal subset.
	\end{proof}
	
	Let $b\in\Bl$;
	recall from Proposition~\ref{prop:eimax-to-chkw} that
	$\te_{\wo}^{\max}(b_\rho\otimes b)=b_\rho\otimes\chke_{\wo}(b)$
	is a highest element of $\B(\rho)\otimes\Bl$.
	Hence there exists a unique isomorphism
	$\Psi_b:C(b_\rho\otimes b)\xrightarrow{\ \sim\ }\B(\mu)$ of crystals
	which maps $b_\rho\otimes\chke_{\wo}(b)$ to $b_\mu$,
	where $\mu:=\wt\bigl(b_\rho\otimes\chke_{\wo}(b)\bigr)\in P^+$.

	\begin{proposition}[cf. Proposition~\ref{prop:unique-checked-rep}]\label{prop:brho-C_b}
		Keep the notation and setting above.
		We have $\Psi_b\bigl(b_\rho\otimes\mathcal{C}_b\bigr)=\B_w(\mu)$ for some $w\in W$.
	\end{proposition}
	\begin{proof}
		We deduce from \cite[Theorem 2.11]{joseph2003} that the image of
		$\bigl(b_\rho\otimes\Bl\bigr)\cap C(b_\rho\otimes b)$ under the isomorphism $\Psi_b$
		is identical to a Demazure crystal $\B_w(\mu)$ for some $w\in W$.
		Hence it suffices to verify that
		$\bigl(b_\rho\otimes\Bl\bigr)\cap C(b_\rho\otimes b)=b_\rho\otimes\mathcal{C}_b$.
		Assume that $c\in\mathcal{C}_b$.
		It follows from Proposition~\ref{prop:eimax-to-chkw} that
		$$\te_{w_\circ}^{\max}(b_\rho\otimes c)=b_\rho\otimes\chke_{w_\circ}(c)=b_\rho\otimes\chke_{w_\circ}(b).$$
		Hence, $b_\rho\otimes\mathcal{C}_b\subseteq\bigl(b_\rho\otimes\Bl\bigr)\cap C(b_\rho\otimes b)$.
		It suffices to show that $b_\rho\otimes d\notin C(b_\rho\otimes b)$ for all $d\in\Bl\setminus\mathcal{C}_b$.
		Suppose, for a contradiction, that 
		$b_\rho\otimes d\in C(b_\rho\otimes b)$.
		Because both $\te_{w_\circ}^{\max}(b_\rho\otimes d)$ and $\te_{w_\circ}^{\max}(b_\rho\otimes b)$ are highest elements in $C(b_\rho\otimes b)$,
		we have $b_\rho\otimes \chke_{w_\circ}(d)=\te_{w_\circ}^{\max}(b_\rho\otimes d)=\te_{w_\circ}^{\max}(b_\rho\otimes b)=b_\rho\otimes \chke_{w_\circ}(b)$.
		Hence we get $\chke_{w_\circ}(d)=\chke_{w_\circ}(b)$,
		and hence $d\in \mathcal{C}_b$,
		which is a contradiction.
		Therefore, we obtain $\bigl(b_\rho\otimes\Bl\bigr)\cap C(b_\rho\otimes b)=b_\rho\otimes\mathcal{C}_b$,
		which completes the proof.
	\end{proof}

	\begin{corollary}
		Let $E\in\mathbb{E}_\lambda$,
		and $w\in W^\lambda$.
		Then, $b_\rho\otimes E$ and $b_\rho\otimes \mathcal{A}_w(\lambda)$ decompose into disjoint unions of Demazure crystals,
		in the sense that $b_\rho\otimes X=\bigsqcup_{b\in X\cap\Tfix}b_\rho\otimes\mathcal{C}_b$ for $X=E,\ \mathcal{A}_w(\lambda)$,
		where each summand $b_\rho\otimes\mathcal{C}_b$ is identified, via the isomorphism $\Psi_b$ of
		Proposition~\ref{prop:brho-C_b}, with a Demazure crystal $\B_{w}(\mu)\subseteq\B(\mu)$ for some $w\in W$
		and $\mu\in P^+$ depending on $b$.
	\end{corollary}
	\begin{proof}
		We deduce this from Propositions~\ref{prop:decomp-into-classes} and \ref{prop:brho-C_b}.
	\end{proof}
	
	\section*{Acknowledgements}
	The author would like to express sincere gratitude to Professor Daisuke Sagaki
	for invaluable guidance and unfailing support throughout this work.
	The author was partly supported by JST SPRING, Grant Number JPMJSP2124.
	
	\renewcommand{\refname}{References}

\end{document}